\documentclass{article}

\usepackage[margin=2cm]{geometry}
\usepackage{amsfonts}
\usepackage{amssymb}
\usepackage{amsmath}
\usepackage{color}
\usepackage{fancyhdr}
\usepackage{amsthm}
\usepackage[hidelinks]{hyperref}
\usepackage{verbatim}
\usepackage{float}
\usepackage[all]{xy}
\usepackage{graphicx}
\usepackage{caption}
\usepackage{subcaption}
\usepackage{mathtools}
\usepackage{youngtab}
\usepackage{tikz}
\usepackage{comment}

\newtheorem{theorem}{Theorem}

\newtheorem{lemma}[theorem]{Lemma}
\newtheorem{proposition}[theorem]{Proposition}

\newtheorem{corollary}[theorem]{Corollary}
\newtheorem{conjecture}[theorem]{Conjecture}
\theoremstyle{definition}
\newtheorem{definition}{Definition}

\newcommand{\bb}{\mathbb}
\newcommand{\mc}{\mathcal}
\newcommand{\Z}{\bb{Z}}
\newcommand{\F}{\mc{F}}
\newcommand{\E}{\bb{E}}
\renewcommand{\P}{\bb{P}}
\newcommand{\Range}{\mathrm{Range}}

\usepackage{titling}
\thanksmarkseries{arabic}

\title{On the Distribution of Range for Tree-Indexed Random Walks}
\author{Aaron Berger\thanks{MIT, \textit{bergera@mit.edu}}, Caleb Ji\thanks{Washington University, St. Louis, \textit{caleb.ji@wustl.edu}}, Erik Metz\thanks{University of Maryland, \textit{emetz@umd.edu}}}
\date{}

\begin{document}

\maketitle

\begin{abstract}
    We study tree-indexed random walks as introduced by Benjamini, H\"aggstr\"om, and Mossel, i.e. labelings of a tree for which adjacent vertices have labels differing by 1. It is a conjecture of those authors that the distribution of the range for any such tree is dominated by that of a path on the same number of edges. The two main variants of this conjecture considered in the literature are the \textit{standard} walks, in which adjacent vertices must have labels differing by \textit{exactly} 1, and \textit{lazy} walks, in which adjacent vertices must have labels differing by \textit{at most} 1. We confirm this conjecture for all trees in the lazy case and provide some partial results in the standard case.
\end{abstract}
\section{Introduction}
In 2000, Benjamini, H\"aggstr\"om, and Mossel \cite{bhm00} began the study of \textit{random graph homomorphisms into $\Z$}, alternatively known as \textit{graph-indexed random walks}. For a graph $G = (V,E)$ with distinguished vertex $v_0 \in E$, the \textit{$G$-indexed walks} are labelings of the following form:
\begin{equation*}
\F(G,v_0) := \left\{ f: V \to \Z \mid f(v_0) = 0, \{u,v\} \in E \implies |f(u) - f(v)| = 1\right\}.
\end{equation*} 
As defined, such walks only exist when $G$ is bipartite, and so Loebl, Ne{\v{s}}et{\v{r}}il, and Reed \cite{lnj03} propose a model in which $\{u,v\} \in E$ implies $|f(u) - f(v)| \le 1$. These labelings are sometimes referred to as \textit{1-Lipschitz functions on graphs}, but here we will refer to them as \textit{lazy random walks} $\mc F'(G,v_0)$ for consistency of terminology. Following Benjamini, H\"aggstr\"om, and Mossel, let $f$ be a $G$-indexed walk chosen uniformly at random from $\F$ (or $\mc F'$, and consider properties such as the expected distance between a fixed pair of vertices and the expected range. Note that both of these quantities are invariant when adding a constant to all labels in a labeling of $G$, and as such are independent of the choice of $v_0$. We can then ignore the information of the choice of $v_0$ and simply refer to the space of labelings as $\F(G)$ (and $\mc F'(G)$).

When comparing walks on different graphs, intuition would suggest that graph-indexed random walks on paths would be likely to have the largest range, and that adding more edges to a graph would necessarily bring vertices closer together in expectation. This second statement is not always true: Benjamini et al. exhibit a graph $G$ with two vertices $u$ and $v$, such that $\bb E(|f(u) - f(v)|)$ actually \textit{increases} upon adding an edge to $G$. Despite this, they also show for any $G, u, v$ that $\bb E(|f(u) - f(v)|)$ increases when $G$ is pared down to any path from $u$ to $v$, as one would expect. Moreover, a stronger statement holds--the distribution is \textit{stochastically dominated} by that of a path:
\begin{theorem}[\cite{bhm00}, Theorem 2.8]
    Let $G = (V,E)$ be a bipartite, connected, finite graph, let $u,v \in V$ and let $P$ be any path from $u$ to $v$ in $G$. Then for all $k$,
    \begin{equation}
    \P_{f \in \F(G)}\left(|f(u)-f(v)| \ge k\right)  \le \P_{f \in \F(P)} \left(|f(u)-f(v)| \ge k\right).
    \end{equation}
\end{theorem}
Stochastic domination is equivalent to stating that for any increasing $g$, the expectation of $g(|f(u)-f(v)|)$ is greater for a path than for any other graph. Taking $g(x) = x$ yields the weaker result that the expected difference between the labels of $u$ and $v$ is larger for a path, but domination also implies that quantities such as the expected squared distance are larger for the path as well.

Loebl, Ne{\v{s}}et{\v{r}}il, and Reed prove a similar, but weaker, result in the lazy random walk setting:
\begin{theorem}[\cite{lnj03}, Theorem 4]
Let $G = (V,E)$ be a connected, finite graph (not necessarily bipartite) with $n$ vertices, let $u,v \in V$ and let $P$ be a path with $n$ vertices. Then for all $k$,
    \begin{equation}
    \E_{f \in \F'(G)}\left(|f(u)-f(v)|\right)  \le \E_{f \in \F'(P)} \left(|f(u)-f(v)|\right).
    \end{equation}
\end{theorem}

For the purposes of this paper we define the range of a graph labeling $\Range(f) := \max_{u,v \in V} f(u) - f(v).$ Benjamini, H\"aggstr\"om, and Mossel make the following two conjectures regarding the range of a graph-indexed random walk:

\begin{conjecture}[\cite{bhm00}, Conjecture 2.10]
Let $G$ be a simple connected graph on $n$ vertices and let $P$ be the path on $n$ vertices. Then:
\begin{itemize}
    \item (Weak). $\E_{f \in \F(G)}\Range(f) \leq \E_{f \in \F(P)}\Range(f)$.
    \item (Strong). $\P_{f \in \F(G)}(Range(f) \ge k) \leq \P_{f \in \F(P)}(Range(f) \ge k)$, for all $k$.
\end{itemize}

\end{conjecture}

The conjecture may be analogously formulated in the lazy walk case. In the literature, there has been some progress made on the weak conjecture, and no progress made on the strong conjecture. Wu, Xu, and Zhu \cite{wxz16} resolve the weak conjecture in the affirmative for trees for both the standard and lazy random walks, and Bok and Ne{\v{s}}et{\v{r}}il \cite{bn18} extend this work to confirm the weak conjecture for unicyclic graphs. Loebl, Ne{\v{s}}et{\v{r}}il, and Reed \cite{lnj03} show that the expected range for any graph is bounded by some absolute constant multiple of the expected range of a path, in the lazy case. The main results of our paper resolve the strong conjecture in the affirmative for all trees in the lazy case, and for \textit{spiders},  trees with (at most) one vertex with degree greater than 2, in the standard case:
\begin{theorem}\label{Main Result 1}
    Let $T$ be a tree on $n$ vertices and $P$ be the path on $n$ vertices. Then for all $k$,
    \begin{equation*}
        \P_{f \in \F'(T)}(Range(f) \ge k) \leq \P_{f \in \F'(P)}(Range(f) \ge k).
    \end{equation*}
\end{theorem}
\begin{theorem}\label{Main Result 2}
    Let $T$ be a spider on $n$ vertices and $P$ be the path on $n$ vertices. Then for all $k$,
    \begin{equation*}
        \P_{f \in \F(T)}(Range(f) \ge k) \leq \P_{f \in \F(P)}(Range(f) \ge k).
    \end{equation*}
\end{theorem}
\subsection{Remarks}
The above definitions and conjectures are even more natural when restricted to trees. In the case of trees on $n$ vertices, there are always $2^{n-1}$ elements of $\F(T)$ (or $3^{n-1}$ in the lazy case), and consequently the computations of probabilities are replaced by enumerations of sets. The case of a tree-indexed random walk had been studied before the introduction of $G$-indexed random walks, although this earlier work was concentrated on infinite trees (for example, \cite{bp942, bp94}). Regarding graph homomorphisms specifically, much of the literature so far has been asymptotic and hence does not provide the exact precision required to show domination of distributions (see \cite{byy07, g03, k01}). In addition to the work mentioned above, Csikv{\'a}ri and Lin \cite{cl14} study random graph homomorphisms from trees into paths, the number of which is counted (in our notation) by $F^k(T)$, a key quantity we work with in the body of this paper.

It may also be worth remarking on the obstacles that prevent extending the result for lazy walks to standard walks. A major problem in the standard case is as follows: let $P_3$ be the path with 3 edges. We are curious about $f_k^2(P_3)$, which we define later to be the number of valid ways to label the vertices with labels in $[0,2]$ such that the first vertex is labeled with $k$ and at least one vertex is labeled $2$. Intuitively, this quantity should increase as $k$ becomes closer to 2. However it does not: there are three labelings when $k=1$ and two when $k=2$. In what is likely a direct consequence of this, another problem arises: Consider the tree with 7 edges given by taking a path of 3 edges and appending a pair of leaves to both endpoints. This tree has two vertices of degree 3, but is not dominated by any other tree with seven edges other than the path of length seven. Consequently, no inductive argument that considers only one high-degree vertex at a time will be sufficient to handle this tree. 

\subsection{Acknowledgments}
The first author would like to thank Jeff Kahn for suggesting this problem to him in 2017. This research was carried out in part at the Duluth REU, which is supported by NSF/DMS grant 1650947 and NSA grant
H98230-18-1-0010, and by the University of Minnesota Duluth.

\section{Preliminaries}\label{Prelims}
We begin with the standard case, where adjacent vertices must have labels differing by exactly 1.
\begin{definition}
For a given tree $T$, let $F^k(T)$ be the number of labelings of $T$ with integers from $0$ to $k$ such that adjacent vertices are labeled with consecutive integers. Such labelings will be referred to as ``valid.'' 
\end{definition}

\begin{definition}
For a given tree $T$, let $f^k(T)$ be the number of labelings of $T$ (with, say, integers from $0$ to $k$) such that adjacent vertices are labeled with connected integers, up to equivalence by translation.
\end{definition}

\noindent
\textit{Remark.} We may now restate the strong range conjecture as: \textit{$f^k(T) \ge f^k(P)$ for all $k$}.

\begin{proposition}
	$f^k(T) = F^k(T) - F^{k-1}(T)$.
\end{proposition}
\begin{proof}
Since every valid labeling bounded by $k-1$ is also a valid labeling bounded by $k$, $F^k(T) - F^{k-1}(T)$ counts the number of valid labelings of $T$ bounded by $k$ that are \textit{not} bounded by $k-1$, i.e. those for which at least one vertex is labeled $k$. Every equivalence class of labelings with range at most $k$ will have exactly one member in this set; simply translate the labeling so the maximum label equals $k$.
\end{proof}

\begin{definition}\label{F_i definition}
	For a tree $T$ with specified root, let $F_i^k(T)$ be the number of labelings of $T$ with labels in $\{0, \ldots, k\}$ such that the root is labeled $i$. 
	
	Let $P_a = \{p_0, p_1, \ldots, p_a\}$ be the path with $a$ edges rooted at its endpoint $p_0$. For paths, let $F_{i \to j}^k(P_a)$ denote the number of valid labelings of $P_a$ such that the label of $p_0$ is $i$, and the label of $p_a$ is $j$. If, for example, $i < 0$, this quantity is simply 0. Similarly, if $i - j \not\equiv a \mod 2$, this quantity will be zero as well.
	
\end{definition}

\noindent
\textit{Remark.} By reflection, we have that $F_i^k(T) = F_{k-i}^k(T)$. For paths in particular, we can condition on whether $p_1 - p_0$ is positive or negative to obtain the recursive formula $F_i^k(P_a) = F_{i+1}^k(P_{a-1}) + F_{i-1}^k(P_{a-1}).$ Similarly, given the first $a-1$ labels, we have no more than 2 choices for the final label, so $F_i^k(P_a) \le 2 F_i^k(P_{a-1})$.

We continue with another intuitive result: a path has more labelings within a bounded interval when its root is closer to the center of that interval. 
\begin{lemma}\label{Center Is Bigger}
	$\left|i - \frac{k}{2}\right| \le \left|j - \frac{k}{2}\right| \Rightarrow F_i^k(P_a) \ge F_j^k(P_a). $
\end{lemma}
\begin{proof}
Note that for $k = 0$ or $1$, the result is trivial. For $k \ge 2$, we proceed by induction on $a$. When $a = 0$, both quantities are 1. Assume now the result holds for $a - 1$. By reflection, it suffices to consider $j \leq i \leq k/2.$
\begin{itemize}
    \item If $i \le k/2 - 1$, then for all $j < i$, we have $j+1 < i+1  \le k/2$ and $j-1 < i-1 \le k/2.$ Inductively, we obtain     
    \begin{equation*}
        F_i^k(P_a) = F_{i+1}^k(P_{a-1}) + F_{i-1}^k(P_{a-1}) \ge F_{j+1}^k(P_{a-1}) + F_{j-1}^k(P_{a-1}) = F_j^k(P_a).
    \end{equation*}
    \item Otherwise, if $k$ is even and $i = k/2$, then 
    \begin{equation*}
    F_i^k(P_a) = F_{k/2+1}^k(P_{a-1}) + F_{k/2-1}^k(P_{a-1}) = 2F_{k/2-1}^k(P_{a-1}) \ge F_{k/2-1}^k(P_{a}) \ge F_j^k(P_a).
    \end{equation*}
    \item Else, if $k$ is odd and $i = k/2-1/2$, then
    \begin{equation*}
    F_i^k(P_a) = F_{k/2+1/2}^k(P_{a-1}) + F_{k/2-3/2}^k(P_{a-1}) = F_{k/2-1/2}^k(P_{a-1}) + F_{k/2-3/2}^k(P_{a-1}) \ge F_{k/2-3/2}^k(P_{a}) \ge F_j^k(P_a).
    \end{equation*}
\end{itemize}
\end{proof}
\begin{corollary}\label{Spider Center is Bigger}
Let $T_{a_1, a_2, \ldots, a_l}$ be the spider with paths of length $a_1, a_2, \ldots, a_l$ emanating from a root.  Then 
\begin{equation*}
\left|i - \frac{k}{2}\right| \le \left|j - \frac{k}{2}\right| \Rightarrow F_i^k(T_{a_1,a_2,a_3,\ldots, a_l}) \ge F_j^k(T_{a_1,a_2,a_3,\ldots, a_l}).
\end{equation*}
\end{corollary}
\begin{proof}
By Lemma \ref{Center Is Bigger}, if $\left|j - \frac{k}{2}\right| \leq \left|i - \frac{k}{2}\right|$, we have
\begin{equation*}
    F_j^k(T_{a_1,a_2,a_3,\ldots, a_l}) = \prod_{t = 1}^l F_j^k(P_{a_t}) \ge \prod_{t = 1}^l F_i^k(P_{a_t}) = F_i^k(T_{a_1,a_2,a_3,\ldots, a_l}).
\end{equation*}
\end{proof}

\textit{Remark.} We will see that this result may be extended to any tree in the lazy case. The fact that there is no clear way to do this in the standard case prevents us from discussing trees other than spiders.

\section{Main Results for Standard Walks}\label{Standard Results}

\begin{lemma}
\label{spidersums}
Let $T_{a_1, a_2, \ldots, a_l}$ be the spider with paths of length $a_1, a_2, \ldots, a_l$ emanating from a root.  Then 
\begin{equation*}
F^k(T_{a_1,a_2,a_3,\ldots, a_l}) - F^k(T_{a_1 + a_2, a_3, \ldots ,a_l}) = \sum_{0 \leq i < j \le k} F_{i \to j}^k(P_{a_1})
\left(F_{i}^k(P_{a_2}) - F_j^k(P_{a_2})\right)\left(F_{i}^k(T_{a_3,\ldots, a_l}) - F_{j}^k(T_{a_3,\ldots, a_l})\right).
\end{equation*}
\end{lemma}

\begin{proof}

We have:
\begin{align*}
F^k(T_{a_1,a_2,\ldots, a_l}) &= \sum_{i=0}^{k} F_i^k(P_{a_1})\cdots F_i^k(P_{a_l}) \\
&= \sum_{i=0}^{k} \sum_{j=0}^{k} F_{i \to j}^k(P_{a_1})F_{i}^k(P_{a_2})F_{i}^k(T_{a_3,\ldots, a_l})\\
&= \sum_{0 \leq i < j \le k} F_{i \to j}^k(P_{a_1})
    \left(F_{i}^k(P_{a_2})F_{i}^k(T_{a_3,\ldots, a_l})
    +F_{j}^k(P_{a_2})F_{j}^k(T_{a_3,\ldots, a_l}) \right).
\end{align*}

and
\begin{align*}
F^k(T_{a_1+a_2,\ldots, a_l}) &= \sum_{i=0}^{k} F_i^k(P_{a_1+a_2})\cdots F_i^k(P_{a_l}) \\
&= \sum_{i=0}^{k} \sum_{j=0}^{k} F_{i \to j}^k(P_{a_1})F_{j}^k(P_{a_2})F_{i}^k(T_{a_3,\ldots, a_l})\\
&= \sum_{0 \leq i < j \le k} F_{i \to j}^k(P_{a_1})
\left(F_{i}^k(P_{a_2})F_{j}^k(T_{a_3,\ldots, a_l})
+F_{j}^k(P_{a_2})F_{i}^k(T_{a_3,\ldots, a_l}) \right).
\end{align*}

Subtracting these equations and factoring yields the desired expression.
\end{proof}
We have now demonstrated that when combining two legs of a spider, $F^k$ increases. This would immediately be sufficient to show that $F^k$ for a spider is smaller than $F^k$ for a path. However we are concerned not directly with $F$, but rather with $f$, given by its partial differences. We continue by reproducing Lemma \ref{Center Is Bigger} for $f$:
\begin{lemma}\label{1D Below Average}
	Let $i < j \leq k$ such that $\frac{i+j}{2} \leq \frac{k}{2}$. Then:
	$$0 \leq F_j^k(P_a) - F_i^k(P_a) \leq F_j^{k+1}(P_a) - F_i^{k+1}(P_a).$$
\end{lemma}
\begin{proof}
	Positivity follows directly from Lemma \ref{Center Is Bigger}. For the second inequality we proceed in a similar manner to the proof of Lemma \ref{Center Is Bigger}: by induction on $a$, with special cases when $\frac{i+j}{2} = \frac{k}{2}$ and $\frac{i+j}{2} = \frac{k-1}{2}$. 

	When $a = 0$ the result is trivial. For our first special case, if $\frac{i+j}{2} = \frac{k}{2}$, then $\left|i - \frac{k}{2}\right| = \left|j - \frac{k}{2}\right| = \frac{j - i}{2}$, and so the left-hand side is zero by symmetry, whereas the right-hand side is non-negative by the Lemma \ref{Center Is Bigger}. In fact, the second inequality (though not the first) still holds when $\frac{i+j}{2} = \frac{k+1}{2}$: the right-hand size is now 0 by symmetry, whereas the left-hand side is non-positive by Lemma \ref{Center Is Bigger}.
	
	It remains to verify the second inequality when $\frac{i+j}{2} \le \frac{k - 1}{2}$. Let $f_j^{k+1}(P_a) := F_j^{k+1}(P_a) - F_j^{k}(P_a)$. We can rewrite the desired inequality as:
	$$
	f_j^{k+1}(P_a) \geq f_i^{k+1}(P_a).
	$$
	Combinatorially, one can show that $f_j^{k+1}(P_a)$ counts the number of paths starting at $j$ of length $a$ such that at least one vertex is labeled $k+1$. Consequently, the following recursive formula holds for $j < k+1$:
	$$
	f_j^{k+1}(P_a) = f_{j+1}^{k+1}(P_{a-1}) + f_{j-1}^{k+1}(P_{a-1}).
	$$
	We continue inductively. Assume the statement holds for paths of length $a-1$, and recall that we know the desired inequality to always be true whenever $\frac{i+j}{2} \in \{\frac k2, \frac{k+1}{2}\}$. When $\frac{i+j}{2} \le \frac{k - 1}{2}$, either $i < 0$ (in which case the statement is trivial) or $j \le k - 1$.
	Inductively, we have shown that $f_{j+1}^{k+1}(P_{a-1}) \ge f_{i+1}^{k+1}(P_{a-1})$, and $f_{j-1}^{k+1}(P_{a-1}) \ge f_{i-1}^{k+1}(P_{a-1})$ whenever $i+1 < j+1 \leq k$ and $\frac{(i+1)+(j+1)}{2} \leq \frac{k+1}{2}$. Both of these conditions are satisfied by our hypotheses, and adding these two inequalities produces our desired statement.
\end{proof}

Following the pattern above, we may now reproduce Lemma \ref{Spider Center is Bigger} for $f$:

\begin{lemma}\label{Bigger Difference}
	Let $T_{a_1, a_2, \ldots, a_l}$ be the spider with paths of length $a_1, a_2, \ldots, a_l$ emanating from a root, and let $i < j \leq k$ such that $\frac{i+j}{2} \leq \frac{k}{2}$. Then:
	\begin{equation*}
	0 \leq F_{j}^k(T_{a_1,\ldots, a_l}) - F_{i}^k(T_{a_1,\ldots, a_l}) \leq F_{j}^{k+1}(T_{a_1,\ldots, a_l}) - F_{i}^{k+1}(T_{a_1,\ldots, a_l}).
	\end{equation*}
\end{lemma}
\begin{proof}
	As before, positivity follows directly from Corollary \ref{Spider Center is Bigger}. For the second inequality, we proceed by induction on $l$. When $l = 1$, this is just Lemma \ref{1D Below Average}. Otherwise, assume this is true for spiders with $l-1$ legs and rewrite the desired inequality as
	$$
	f_j^{k+1}(T_{a_1,\ldots, a_l}) \geq f_i^{k+1}(T_{a_1,\ldots, a_l}).
	$$
	
	Combinatorially, we have that $f_j^{k+1}(T_{a_1,\ldots, a_l})$ counts the number of trees of the given form with root labeled $j$, such that at least one vertex is labeled $k+1$. Thus we have either a vertex along $P_{a_1}$ labeled $k+1$, a vertex along one of the remaining paths labeled $k+1$, or both:
	$$
	f_j^{k+1}(T_{a_1,\ldots, a_l}) = f_j^{k+1}(P_{a_1})F_j^{k}(T_{a_2,\ldots, a_l}) + F_j^{k}(P_{a_1})f_j^{k+1}(T_{a_2,\ldots, a_l}) + f_j^{k+1}(P_{a_1})f_j^{k+1}(T_{a_2,\ldots, a_l}).
	$$
	Applying Lemma \ref{Center Is Bigger}, Corollary \ref{Spider Center is Bigger}, Lemma \ref{1D Below Average}, and the induction hypothesis as appropriate, we see each term becomes smaller when $j$ is replaced by $i$, which completes the proof.
\end{proof}
\begin{corollary}\label{Also Bigger}
		Let $T = T_{a_1, a_2, \ldots, a_l}$ be the spider with paths of length $a_1, a_2, \ldots, a_l$ emanating from a root, and let $i < j \leq k$ such that $\frac{i+j}{2} \geq \frac{k}{2}$. Then
	\begin{equation*}
	0 \leq F_{i}^k(T) - F_{j}^k(T) \leq F_{i+1}^{k+1}(T) - F_{j+1}^{k+1}(T).
	\end{equation*}
\end{corollary}
\begin{proof}
	Positivity once again follows directly from Corollary \ref{Spider Center is Bigger}. For the second inequality, recall that reflection implies $F_{i}^k(T) = F_{k-i}^k(T)$. We see that the pair $(k - j, k - i)$ satisfies the hypotheses of Lemma \ref{Bigger Difference}, which yields:
	$$
	F_{i}^k(T) - F_{j}^k(T) = F_{k-i}^k(T) - F_{k-j}^k(T) \leq F_{k-i}^{k+1}(T) - F_{k-j}^{k+1}(T) = F_{i+1}^{k+1}(T) - F_{j+1}^{k+1}(T)
	$$
\end{proof}

We now prove the main result of this section.


\begin{proof}[\bf Proof of Theorem \ref{Main Result 2}]
We will combine one pair of legs of the spider at a time to inductively arrive at a path. Since $f^k(T)=F^k(T)-F^{k-1}(T)$ for any tree $T$, it suffices to show that 
\begin{equation*}
F^k(T_{a_1,a_2,a_3,\ldots, a_l}) - F^k(T_{a_1 + a_2, a_3, \ldots ,a_l}) \leq F^{k+1}(T_{a_1,a_2,a_3,\ldots, a_l}) - F^{k+1}(T_{a_1 + a_2, a_3, \ldots ,a_l}).
\end{equation*}

From Lemma \ref{spidersums}, we have 
\begin{equation*}
    F^k(T_{a_1,a_2,a_3,\ldots, a_l}) - F^k(T_{a_1 + a_2, a_3, \ldots ,a_l}) = 
    \sum_{0 \leq i < j \le k} 
    F_{i \to j}^k(P_{a_1})
    \left(F_{i}^k(P_{a_2}) - F_j^k(P_{a_2})\right)
    \left(F_{i}^k(T_{a_3,\ldots, a_l}) - F_{j}^k(T_{a_3,\ldots, a_l})\right).
\end{equation*}

We now look at what happens to each term when we increase $k$ to $k+1$. For convenience we write $T' :=T_{a_3,\ldots,a_\ell}$. When $\frac{i+j}{2} \le \frac{k}{2}$, we compare to the $(i,j)$ summand for $k+1$, and claim:
$$
 F_{i \to j}^k(P_{a_1})
    \left(F_{i}^k(P_{a_2}) - F_j^k(P_{a_2})\right)
    \left(F_{i}^k(T') - F_{j}^k(T')\right) \le  F_{i \to j}^{k+1}(P_{a_1})
    \left(F_{i}^{k+1}(P_{a_2}) - F_j^{k+1}(P_{a_2})\right)
    \left(F_{i}^{k+1}(T') - F_{j}^{k+1}(T')\right)
$$
We prove this inequality term-by-term. By definition $F_{i \to j}^k(P_{a_1}) \le F_{i \to j}^{k+1}(P_{a_1})$. By Lemma \ref{Bigger Difference}, we have that both $\left(F_{i}^k(P_{a_2}) - F_j^k(P_{a_2})\right)$ and $\left(F_{i}^k(T_{a_3,\ldots, a_l}) - F_{j}^k(T_{a_3,\ldots, a_l})\right)$ are negative, and decrease when $k$ is replaced by $k+1$. Consequently the summand is positive and increases. 

On the other hand, when $\frac{i+j}{2} > \frac{k}{2}$, we compare to the $(i+1, j+1)$ summand for $k+1$, i.e. we claim:
$$
 F_{i \to j}^k(P_{a_1})
    \left(F_{i}^k(P_{a_2}) - F_j^k(P_{a_2})\right)
    \left(F_{i}^k(T') - F_{j}^k(T')\right) \le 
    F_{i+1 \to j+1}^{k+1}(P_{a_1})
    \left(F_{i+1}^{k+1}(P_{a_2}) - F_{j+1}^{k+1}(P_{a_2})\right)
    \left(F_{i+1}^{k+1}(T') - F_{j+1}^{k+1}(T')\right)
$$
Again, the proof proceeds term-by-term. We have $F_{i \to j}^k(P_{a_1}) = F_{(k-i) \to (k-j)}^k(P_{a_1}) \le F_{(k-i) \to (k-j)}^{k+1}(P_{a_1}) = F_{i+1 \to j+1}^{k+1}(P_{a_1})$. From Corollary \ref{Also Bigger} we have that both $\left(F_{i}^k(P_{a_2}) - F_j^k(P_{a_2})\right)$ and $\left(F_{i}^k(T_{a_3,\ldots, a_l}) - F_{j}^k(T_{a_3,\ldots, a_l})\right)$ are positive, and increase when $k$ is replaced by $k+1$ and $(i,j)$ by $(i+1, j+1)$. Consequently the summand is again positive, and increases when moving from $k$ to $k+1$, which completes the proof.
\end{proof}

\clearpage
\section{The Lazy Random Walk Model}
The results in Sections \ref{Prelims} and \ref{Standard Results} may be converted easily to the lazy case with few modifications. Abusing notation, in this section we will instead let $F^k(T)$ be the number of labelings of $T$ with integers from $0$ to $k$ such that adjacent vertices are labeled with either consecutive or identical integers. 

\begin{lemma}\label{lazy center is bigger}
For any tree $T$ rooted at a vertex $v_0$, $|i-\frac k2| \le |j - \frac k2| \Rightarrow F_i^k(T) \ge F_j^k(T)$.
\end{lemma}
Compare this to Lemma \ref{Center Is Bigger} and Corollary \ref{Spider Center is Bigger} of the standard case. In the lazy case, we are able to easily extend this result to all trees, whereas in the standard case, there are some trees for which this statement is simply false (e.g. a star rooted at a leaf).
\begin{proof}
If $k = 0$ or $1$ the result is trivial. Otherwise, assume $k \ge 2$ and proceed inductively on $|V(T)|$. If $\deg(v_0) > 1$, then we can write $T = T_1 \cup T_2$ the union of two, nonempty trees rooted at $v_0$ that only overlap at $v_0$. In this case we have inductively
$$
F_i^k(T) = F_i^k(T_1)F_i^k(T_2) \ge F_j^k(T_1)F_j^k(T_2) = F_j^k(T).
$$
Otherwise, $\deg(v_0) = 1$. Then there is a unique edge $e = (v_0,v_1)$, and we consider labelings of the subtree $T'$ induced by deleting $v_0$. Without loss of generality assume $j+1 \le i \le \frac k2$. Then by triangle inequality, we have the following inequalities:
$$
\left|i - \frac k2\right| \le \left|j+1 - \frac k2\right|,\quad \left|i+1 - \frac k2\right| \le \left|j - \frac k2\right|,\quad \left|i-1 - \frac k2\right| \le \left|j-1 - \frac k2\right|.
$$
Matching terms and applying the inductive hypothesis, we obtain:
$$
F_i^k(T) = F_{i-1}^k(T')+F_{i}^k(T')+F_{i+1}^k(T') \ge  F_{j-1}^k(T')+F_{j+1}^k(T')+F_{j}^k(T') = F_j^k(T).
$$
\end{proof}

We continue by reproducing Lemmas \ref{1D Below Average} and \ref{Bigger Difference}, as well as Corollary \ref{Also Bigger} in the lazy case. Once again, we are now able to prove these statements for all trees.

\begin{lemma}\label{lazy f center is better}
Let $i < j \le k$ such that $\frac{i+j}2 \le \frac k2$. Then for any tree $T$ rooted at a vertex $v_0$, 
$$
0 \le F_j^k(T) - F_i^k(T) \le  F_j^{k+1}(T) - F_i^{k+1}(T).
$$
\end{lemma}

\begin{proof}
Positivity follows directly from Lemma \ref{lazy center is bigger}, as well as the second inequality if either $\frac{i+j}{2} = \frac k2$ or $\frac{i+j}{2} = \frac{k+1}{2}$: in the former case the left-hand side is 0 and the right-hand side is non-negative, whereas in the latter case the right-hand side is 0 and the left-hand side is non-positive. It remains to prove the second inequality when $\frac{i+j}{2} \le \frac{k-1}{2}$. As before, proceed inductively on the size of $T$, and rewrite the inequality as:
$$
f_j^{k+1}(T) \ge f_i^{k+1}(T).
$$
If $\deg(v_0) > 1$, then we can write $T = T_1 \cup T_2$ the union of two nonempty trees rooted at $v_0$ that only overlap at $v_0$. In this case the inductive hypothesis and Lemma \ref{lazy center is bigger} allow us to compare term-by-term:
\begin{align*}
f_i^{k+1}(T) &= f_i^{k+1}(T_1)F_i^{k+1}(T_2)+F_i^{k+1}(T_1)f_i^{k+1}(T_2)+f_i^{k+1}(T_1)f_i^{k+1}(T_2)\\ &\ge f_j^{k+1}(T_1)F_j^{k+1}(T_2)+F_j^{k+1}(T_1)f_j^{k+1}(T_2)+f_j^{k+1}(T_1)f_j^{k+1}(T_2)\\ &= f_j^{k+1}(T).
\end{align*}
Otherwise, $\deg(v_0) = 1$. Then there is a unique edge $e = (v_0,v_1)$, and we consider labelings of the subtree $T'$ induced by deleting $v_0$. Via the combinatorial interpretation of $f$, for $i < k+1$ we have $f_i(T) = f_{i-1}(T')+f_i(T')+f_{i+1}(T')$. Since $\frac{i+1+j+1}{2} \le {k-1+2}{2} - {k+1}{2}$ by assumption, we can apply induction and match terms to obtain:
$$
f_i^{k+1}(T) = f_{i-1}^{k+1}(T')+f_{i}^{k+1}(T')+f_{i+1}^{k+1}(T') \ge  f_{j-1}^{k+1}(T')+f_{j}^{k+1}(T')+f_{j}^{k+1}(T') = f_j^{k+1}(T).
$$
\end{proof}
\begin{corollary}\label{lazy bigger too}
Let $T$ be a tree and let $i < j \le k$ such that $\frac{i+j}{2} \ge \frac k2$. Then:
$$
0 \le F_i^k(T) - F_j^k(T) \le F_{i+1}^{k+1}(T) - F_{j+1}^{k+1}(T).
$$
\end{corollary}
\begin{proof}
The symmetry of $F$ by reflection implies $F_i^k(T) = F_{k-i}^k(T)$. We see that the pair $(k-j,k-i)$ satisfies the hypotheses of Lemma \ref{lazy f center is better}, and so:
$$
F_i^k(T) - F_j^k(T) = F_{k-i}^k(T) - F_{k-j}^k(T) \le F_{k-i}^{k+1}(T) - F_{k-j}^{k+1}(T) = F_{i+1}^{k+1}(T) - F_{j+1}^{k+1}(T).
$$
\end{proof}


\subsection{Completing the proof for lazy walks}
\begin{proof}[Proof of Theorem \ref{Main Result 1}] Let $v_0$ be a vertex of $V$ with two bare paths emanating from it, so we can write $T = P_a \cup P_b \cup T'$ where $T'$ is the leftover vertices and edges not in $P_a$ or $P_b$. Inductively, it will suffice to demonstrate that $f_{k+1}(T) \ge f_{k+1}(T' \cup P_{a+b}$, as the process of combining the two paths $P_a$ and $P_b$ into one results in a tree with one fewer leaf, and therefore will eventually terminate in a path. Rewriting this inequality in terms of $F$, we want to demonstrate:
$$
F^k(T) - F^k(T' \cup P_{a+b}) \le F^{k+1}(T) - F^{k+1}(T' \cup P_{a+b}).
$$
We may expand both sides of this inequality in the same manner as Lemma \ref{spidersums}, e.g. the left-hand side becomes:
$$
\sum_{0 \leq i < j \le k} F_{i \to j}^k(P_{a_1})
\left(F_{i}^k(P_{a_2}) - F_j^k(P_{a_2})\right)\left(F_{i}^k(T') - F_{j}^k(T')\right).
$$
We now investigate what happens to a single summand when we move from $k$ to $k+1$.

When $\frac{i+j}{2} \le \frac{k}{2}$, we compare to the $(i,j)$ summand for $k+1$, and claim:
$$
 F_{i \to j}^k(P_{a_1})
    \left(F_{i}^k(P_{a_2}) - F_j^k(P_{a_2})\right)
    \left(F_{i}^k(T') - F_{j}^k(T')\right) \le  F_{i \to j}^{k+1}(P_{a_1})
    \left(F_{i}^{k+1}(P_{a_2}) - F_j^{k+1}(P_{a_2})\right)
    \left(F_{i}^{k+1}(T') - F_{j}^{k+1}(T')\right)
$$

By definition $F_{i \to j}^k(P_{a_1}) \le F_{i \to j}^{k+1}(P_{a_1})$. By Lemma \ref{lazy f center is better}, we have that both $\left(F_{i}^k(P_{a_2}) - F_j^k(P_{a_2})\right)$ and $\left(F_{i}^k(T') - F_{j}^k(T')\right)$ are negative, and decrease when $k$ is replaced by $k+1$. Consequently the summand is positive and increases. 

On the other hand, when $\frac{i+j}{2} > \frac{k}{2}$, we compare to the $(i+1, j+1)$ summand for $k+1$, i.e. we claim:
$$
 F_{i \to j}^k(P_{a_1})
    \left(F_{i}^k(P_{a_2}) - F_j^k(P_{a_2})\right)
    \left(F_{i}^k(T') - F_{j}^k(T')\right) \le 
    F_{i+1 \to j+1}^{k+1}(P_{a_1})
    \left(F_{i+1}^{k+1}(P_{a_2}) - F_{j+1}^{k+1}(P_{a_2})\right)
    \left(F_{i+1}^{k+1}(T') - F_{j+1}^{k+1}(T')\right)
$$
Again, the proof proceeds term-by-term. We have $F_{i \to j}^k(P_{a_1}) = F_{(k-i) \to (k-j)}^k(P_{a_1}) \le F_{(k-i) \to (k-j)}^{k+1}(P_{a_1}) = F_{i+1 \to j+1}^{k+1}(P_{a_1})$. From Corollary \ref{lazy bigger too} we have that both $\left(F_{i}^k(P_{a_2}) - F_j^k(P_{a_2})\right)$ and $\left(F_{i}^k(T') - F_{j}^k(T')\right)$ are positive, and increase when $k$ is replaced by $k+1$ and $(i,j)$ by $(i+1, j+1)$. Consequently the summand is again positive, and increases when moving from $k$ to $k+1$, which completes the proof.

\end{proof}

\bibliographystyle{acm}

\end{document}